\documentclass[smallextended,numbook,runningheads]{svjour3}     
\smartqed  
\usepackage{graphicx}
\usepackage{amsmath}

\usepackage{mathptmx}      
\usepackage{amsfonts,amssymb,mathdots}
\usepackage{algorithm,algpseudocode}
\newcommand{\ten}[1]{\mathcal{#1}}
\journalname{BIT}

\begin{document}

\title{Inheritance Properties and Sum-of-Squares Decomposition of Hankel Tensors: Theory and Algorithms\thanks{The first and the third authors are supported by the National Natural Science Foundation of China under grant 11271084. The second author is supported by the Hong Kong Research Grant Council (Grant No. PolyU 502111, 501212, 501913 and 15302114).
}}

\titlerunning{Hankel Tensors}        

\author{Weiyang Ding         \and
        Liqun Qi             \and
        Yimin Wei
}

\authorrunning{W. Ding, L. Qi, and Y. Wei} 

\institute{W. Ding \at
              School of Mathematical Sciences, Fudan University, Shanghai 200433, China \\
              \email{dingw11@fudan.edu.cn, weiyang.ding@gmail.com}           
           \and
           L. Qi \at
              Department of Applied Mathematics, the Hong Kong Polytechnic University, Hong Kong \\
              \email{liqun.qi@polyu.edu.hk}
           \and
           Y. Wei \at
              School of Mathematical Sciences and Shanghai Key Laboratory of
Contemporary Applied Mathematics, Fudan University, Shanghai 200433, China \\
              \email{ymwei@fudan.edu.cn, yimin.wei@gmail.com}
}

\date{Received: date / Accepted: date}

\maketitle

\begin{abstract}
  In this paper, we show that if a lower-order Hankel tensor is positive semi-definite (or positive definite, or negative semi-definite, or negative definite, or SOS), then its associated higher-order Hankel tensor with the same generating vector, where the higher order is a multiple of the lower order, is also positive semi-definite (or positive definite, or negative semi-definite, or negative definite, or SOS, respectively).  Furthermore, in this case, the extremal H-eigenvalues of the higher order tensor are bounded by the extremal H-eigenvalues of the lower order tensor, multiplied with some constants. Based on this inheritance property, we give a concrete sum-of-squares decomposition for each strong Hankel tensor. Then we prove the second inheritance property of Hankel tensors, i.e., a Hankel tensor has no negative (or non-positive, or positive, or nonnegative) H-eigenvalues if the associated Hankel matrix of that Hankel tensor has no negative (or non-positive, or positive, or nonnegative, respectively) eigenvalues.  In this case, the extremal H-eigenvalues of the Hankel tensor are also bounded by the extremal eigenvalues of the associated Hankel matrix, multiplied with some constants. The third inheritance property of Hankel tensors is raised as a conjecture.

\end{abstract}

\section{Introduction}

Hankel structures are widely employed in data analysis and signal processing. Not only Hankel matrices but also higher-order Hankel tensors arise frequently in many disciplines such as exponential data fitting \cite{Boyer07,Ding15,Papy05,Papy09}, frequency domain subspace identification \cite{Smith14}, multidimensional seismic trace interpolation \cite{Trickett13}, and so on. Furthermore, the positive semi-definite Hankel matrices are most related with the moment problems, and one can refer to \cite{Akhiezer65,Fasino95,Shohat43}. In moment problems, some necessary or sufficient conditions for the existence of a desired measure are given as the positive semi-definiteness of a series of certain Hankel matrices.

An $m^{\rm th}$-order \emph{Hankel tensor} $\ten{H} \in \mathbb{C}^{n_1 \times n_2 \times \dots \times n_m}$ is a multidimensional array whose entries obey a function of the sums of indices, i.e.,
$$
\ten{H}_{i_1,i_2,\dots,i_m} = h_{i_1+i_2+\dots+i_m},\ i_k = 0,1,\dots,n_k-1,\ k = 1,2,\dots,m,
$$
where the vector ${\bf h}$ is called the \emph{generating vector} of this Hankel tensor $\ten{H}$ \cite{Ding15,Li14c,Qi15,Xu15}. Actually, the generating vector and the size parameters totally determine this Hankel tensor. The multiplication of a tensor $\ten{T}$ and a matrix $M$ on the $k^{\rm th}$ mode (see \cite[Chapter 12.4]{Golub13}) is defined by
$$
(\ten{T} \times_k M)_{i_1 \dots i_{k-1} j_k i_{k+1} \dots i_m} := \sum_{i_k=0}^{n-1} \ten{T}_{i_1 i_2 \dots i_m} M_{i_k j_k}.
$$
When the matrix degrades into a vector, Qi \cite{Qi05} introduced some simple but useful notations
\[
\begin{split}
\ten{T} {\bf x}^m &:= \ten{T} \times_1 {\bf x} \times_2 {\bf x} \dots \times_m {\bf x}, \\
\ten{T} {\bf x}^{m-1} &:= \ten{T} \hspace{20pt} \times_2 {\bf x} \dots \times_m {\bf x}.
\end{split}
\]
Each $m^{\rm th}$-order $n$-dimensional square tensor can induce a degree-$m$ multivariate polynomial of $n$ variables
$$
p_{\ten{T}}({\bf x}) := \ten{T} {\bf x}^m = \sum_{i_1,\dots,i_m=0}^{n-1} \ten{T}_{i_1 i_2 \dots i_m} x_{i_1} x_{i_2} \dots x_{i_m}.
$$
Suppose that $m$ is even.  If $p_{\ten{T}}({\bf x})$ is always nonnegative (or positive, or non-positive, or negative) for all nonzero real vectors ${\bf x}$, then  the tensor $\ten{T}$ is called a \emph{positive semi-definite tensor} (or \emph{positive definite tensor}, or \emph{negative semi-definite tensor}, or \emph{negative definite tensor}, respectively) (see \cite{Qi05}). If $p_{\ten{T}}({\bf x})$ can be represented as a sum of squares, then we call the tensor $\ten{T}$ an \emph{SOS tensor} (see \cite{Luo15}). Apparently, an SOS tensor must be positive semi-definite, but the converse is generally not true. If we restrict on Hankel tensors, then it is still an open question whether a positive semi-definite Hankel tensor is also SOS.

Qi \cite{Qi15} showed an inheritance property of Hankel tensors.  The generating vector of a Hankel tensor also generates a Hankel matrix, which is called the associated Hankel matrix of that Hankel tensor \cite{Qi15}.   It was shown in \cite{Qi15} that if the Hankel tensor is of even order and its associated Hankel matrix is positive semi-definite, then the Hankel tensor is also positive semi-definite. In \cite{Qi15}, a Hankel tensor is called a strong Hankel tensor if its associated Hankel matrix is positive semi-definite.   Thus, the above result is that an even order strong Hankel tensor is positive semi-definite. In this paper, we explore more inheritance properties of Hankel tensors.

The first inheritance property of Hankel tensors we will establish in this paper is that if a lower-order Hankel tensor is positive semi-definite (or positive definite, or negative semi-definite, or negative definite, or SOS), then its associated higher-order Hankel tensor with the same generating vector, where the higher order is a multiple of the lower order, is also positive semi-definite (or positive definite, or negative semi-definite, or negative definite, or SOS, respectively).    The inheritance property established in \cite{Qi15} can be regarded as a special case of this inheritance property.  Furthermore, in this case, we show that the extremal H-eigenvalues of the higher order tensor are bounded by the extremal H-eigenvalues of the lower order tensor, multiplied with some constants.

In \cite{Li14c}, it was proved that strong Hankel tensors are SOS tensors, but no concrete SOS decomposition was given there.   In this paper, by using the inheritance property described above, we give a concrete sum-of-squares decomposition for each strong Hankel tensor.

The second inheritance property of Hankel tensors we will establish in this paper is an extension of the inheritance property established in \cite{Qi15} to the odd-order case.
In the common sense, positive semi-definiteness and SOS property are only well-defined for even order tensors.  By \cite{Qi05}, an even order symmetric tensor is positive semi-definite if and only if it has no negative H-eigenvalues.   In this paper, we will show that if the associated Hankel matrix of a Hankel tensor has no negative (or non-positive, or positive, or nonnegative) eigenvalues, then the Hankel tensor has also no negative (or non-positive, or positive, or nonnegative, respectively) H-eigenvalues.   In this case, we show that the extremal H-eigenvalues of the Hankel tensor are also bounded by the extremal eigenvalues of the associated Hankel matrix, multiplied with some constants.

Finally, we raise the third inheritance property of Hankel tensors as a conjecture.

This paper is organized as follows. In Section 2, we first introduce some basic concepts and properties of Hankel tensors. Then by using a convolution formula, we show that if a lower-order Hankel tensor is positive semi-definite  (or positive definite, or negative semi-definite, or negative definite, or SOS), then its associated higher-order Hankel tensor with the same generating vector and a multiple order, is also positive semi-definite  (or positive definite, or negative semi-definite, or negative definite, or SOS, respectively).  In this case, some inequalities to bound the extremal H-eigenvalues of the higher order tensor by the extremal H-eigenvalues of the lower order tensor, multiplied with some constants, are given.
Based on this inheritance property, we give a concrete sum-of-squares decomposition for each strong Hankel tensor. In Section 3, we investigate some structure-preserving Vandermonde decompositions of some particular Hankel tensors, and we prove that each strong Hankel tensor admits an augmented Vandermonde decomposition with all positive coefficients.   With this tool, we show that if the associated Hankel matrix of a Hankel tensor has no negative (or non-positive, or positive, or nonnegative) eigenvalues, then the Hankel tensor has also no negative (or non-positive, or positive, or nonnegative, respectively) H-eigenvalues, i.e., the second inheritance property of Hankel tensors holds. In this case, we show that the extremal H-eigenvalues of the Hankel tensor are also bounded by the extremal eigenvalues of the associated Hankel matrix, multiplied with some constants.  Numerical examples are given in Section 4.  The third inheritance property of Hankel tensors is raised in Section 5 as a conjecture.

\section{The First Inheritance Property of Hankel Tensors}

This section is devoted to the first inheritance property of Hankel tensors. We will prove that if a lower-order Hankel tensor is positive semi-definite or SOS, then a Hankel tensor with the same generating vector but a high multiple order is also positive semi-definite or SOS, respectively.

\subsection{Hankel Tensor-Vector Products}

We shall have a close look at the nature of the Hankel structure first. Ding et al. \cite{Ding15} proposed a fast scheme for multiplying a Hankel tensor by vectors. The main approach is embedding a Hankel tensor into a larger anti-circulant tensor, which can be diagonalized by the Fourier matrices. A special $m^{\rm th}$-order $N$-dimensional Hankel tensor $\ten{C}$ is called an \emph{anti-circulant tensor}, if its generating vector has a period $N$. Let the first $N$ components of its generating vector be ${\bf c} = [c_0,c_1,\dots,c_{N-1}]^\top$. Then the generating vector of $\ten{C}$ has the form
$$
[c_0,c_1,\dots,c_{N-1},\dots,c_0,c_1,\dots,c_{N-1},c_0,c_1,\dots,c_{N-m}]^\top \in \mathbb{C}^{m(N-1)+1}.
$$
Thus we often call the vector ${\bf c}$ the \emph{compressed generating vector} of the anti-circulant tensor $\ten{C}$. Ding et al. proved in \cite[Theorem 3.1]{Ding15} that an $m^{\rm th}$-order $N$-dimensional anti-circulant tensor can be diagonalized by the $N$-by-$N$ Fourier matrix, i.e.,
$$
\ten{C} = \ten{D} \times_1 F_N \times_2 F_N \dots \times_m F_N,
$$
where $F_N = \big(\exp (\frac{2\pi\imath}{N}jk)\big)_{j,k=0}^{N-1}$ ($\imath = \sqrt{-1}$) is the $N$-by-$N$ Fourier matrix, and $\ten{D}$ is a diagonal tensor with diagonal entries ${\rm ifft}({\bf c}) = F_N^\ast {\bf c}$. (Here, ``{\rm ifft}'' is an abbreviation of ``inverse fast Fourier transform''.) Then given $m$ vectors ${\bf y}_1, {\bf y}_2, \dots, {\bf y}_m \in \mathbb{C}^{N}$, we can calculate the anti-circulant tensor-vector product by
$$
\ten{C} \times_1 {\bf y}_1 \times_2 {\bf y}_2 \dots \times_m {\bf y}_m = (F_N^\ast {\bf c})^\top \big(F_N {\bf y}_1 .\ast F_N {\bf y}_2 .\ast \cdots .\ast F_N {\bf y}_m\big),
$$
where ``$.\ast$" is a Matlab-type notation for multiplying two vectors component-by-component, and $F_N {\bf y}_k$ and $F_N^\ast {\bf c}$ can be realized via ${\rm fft}$ and ${\rm ifft}$, respectively.

Let $\ten{H}$ be an $m^{\rm th}$-order Hankel tensor of size $n_1 \times n_2 \times \dots \times n_m$ and ${\bf h}$ be its generating vector. Taking the vector ${\bf h}$ as the compressed generating vector, we can also form an anti-circulant tensor $\ten{C}_{\ten{H}}$ of order $m$ and dimension $N = n_1+ \dots + n_m -m + 1$. Interestingly, we find out that the Hankel tensor $\ten{H}$ is exactly the first leading principal subtensor of $\ten{C}_{\ten{H}}$, that is, $\ten{H} = \ten{C}_{\ten{H}}(1:n_1,1:n_2,\dots,1:n_m)$. Hence, the Hankel tensor-vector product $\ten{H} \times_1 {\bf x}_1 \times_2 {\bf x}_2 \dots \times_m {\bf x}_m$ is equal to the anti-circulant tensor-vector product
$$
\ten{C}_{\ten{H}} \times_1 \begin{bmatrix} {\bf x}_1 \\ {\bf 0} \end{bmatrix} \times_2 \begin{bmatrix} {\bf x}_2 \\ {\bf 0} \end{bmatrix} \dots \times_m \begin{bmatrix} {\bf x}_m \\ {\bf 0} \end{bmatrix},
$$
where ${\bf 0}$ denotes an all-zero vector of proper size. Thus it can be computed via
$$
\ten{H} \times_1 {\bf x}_1 \times_2 {\bf x}_2 \dots \times_m {\bf x}_m = (F_N^\ast {\bf h})^\top \bigg(F_N \begin{bmatrix} {\bf x}_1 \\ {\bf 0} \end{bmatrix} .\ast F_N \begin{bmatrix} {\bf x}_2 \\ {\bf 0} \end{bmatrix} .\ast \cdots .\ast F_N \begin{bmatrix} {\bf x}_m \\ {\bf 0} \end{bmatrix}\bigg).
$$
Particularly, when $\ten{H}$ is square and all the vectors are the same, i.e., $n := n_1 = \dots = n_m$ and ${\bf x} := {\bf x}_1 = \dots = {\bf x}_m$, the homogeneous polynomial can be evaluated via
\begin{equation}\label{eq_tvp}
\ten{H} {\bf x}^m = (F_N^\ast {\bf h})^\top \bigg(F_N \begin{bmatrix} {\bf x} \\ {\bf 0} \end{bmatrix}\bigg)^{[m]},
\end{equation}
where $N = mn-m+1$, and ${\bf v}^{[m]} = [v_1^m, v_2^m, \dots, v_N^m]^\top$ stands for the componentwise $m^{\rm th}$ power of the vector ${\bf v}$. Moreover, this scheme has an analytic interpretation.

\subsection{A Convolution Formula}

In matrix theory, the multiplications of most structured matrices, such as Toeplitz, Hankel, Vandermonde, Cauchy, etc., with vectors have their own analytic interpretations. Olshevsky and Shokrollahi \cite{Olshevsky01} listed several important connections between fundamental analytic algorithms and structured
matrix-vector multiplications. They claimed that there is a close relationship between Hankel matrices and discrete convolutions. And we will see shortly that this is also true for Hankel tensors.

We first introduce some basic facts about discrete convolutions. Let two vectors ${\bf u} \in \mathbb{C}^{n_1}$ and ${\bf v} \in \mathbb{C}^{n_2}$. Then their \emph{convolution} ${\bf w} = {\bf u} \ast {\bf v} \in \mathbb{C}^{n_1+n_2-1}$ is a longer vector defined by
$$
w_k = \sum_{j=0}^k u_j v_{k-j} \quad \text{for } k = 0,1,\dots,n_1+n_2-2,
$$
where $u_j = 0$ when $j \geq n_1$ and $v_j = 0$ when $j \geq n_2$. Denote $p_{\bf u}(\xi)$ and $p_{\bf v}(\xi)$ as the polynomials whose coefficients are ${\bf u}$ and ${\bf v}$, respectively, i.e.,
$$
p_{\bf u}(\xi) = u_0 + u_1 \xi + \dots + u_{n_1-1} \xi^{n_1-1}, \quad p_{\bf v}(\xi) = v_0 + v_1 \xi + \dots + v_{n_2-1} \xi^{n_2-1}.
$$
Then we can verify easily that ${\bf u} \ast {\bf v}$ consists of the coefficients of the product $p_{\bf u}(\xi) \cdot p_{\bf v}(\xi)$. Another important property of discrete convolutions is that
$$
{\bf u} \ast {\bf v} = V^{-1}\bigg(V\begin{bmatrix}{\bf u}\\{\bf 0}\end{bmatrix} .\ast V\begin{bmatrix}{\bf v}\\{\bf 0}\end{bmatrix}\bigg)
$$
for an arbitrary $(n_1+n_2-1)$-by-$(n_1+n_2-1)$ nonsingular Vandermonde matrix $V$. In applications, the Vandermonde matrix is often taken as the Fourier matrices, since we have fast algorithms for discrete Fourier transforms.

Similarly, if there are more vectors ${\bf u}_1,{\bf u}_2,\dots,{\bf u}_m$, then their convolution is equal to
\begin{equation}\label{eq_conv}
{\bf u}_1 \ast {\bf u}_2 \ast \dots \ast {\bf u}_m = V^{-1}\bigg(V\begin{bmatrix}{\bf u}_1\\{\bf 0}\end{bmatrix} .\ast V\begin{bmatrix}{\bf u}_2\\{\bf 0}\end{bmatrix} .\ast \cdots .\ast V\begin{bmatrix}{\bf u}_m\\{\bf 0}\end{bmatrix}\bigg),
\end{equation}
where $V$ is a nonsingular Vandermonde matrix. Comparing \eqref{eq_tvp} and \eqref{eq_conv}, we can write immediately that
\begin{equation}\label{eq_tvp_conv}
\ten{H} {\bf x}^m = {\bf h}^\top (\underbrace{{\bf x} \ast {\bf x} \ast \dots \ast {\bf x}}_m) =: {\bf h}^\top {\bf x}^{\ast m}.
\end{equation}
Employing this convolution formula for Hankel tensor-vector products, we can derive the inheritability of positive semi-definiteness and SOS property of Hankel tensors from the lower-order case to the higher-order case.

\subsection{Lower-Order Implies Higher-Order}

Use $\ten{H}_m$ to denote an $m^{\rm th}$-order $n$-dimensional Hankel tensor with the generating vector ${\bf h} \in \mathbb{R}^{mn-m+1}$, where $m$ is even and $n = qk-q+1$ for some integers $q$ and $k$. Then by the convolution formula \eqref{eq_tvp_conv}, we have $\ten{H}_m {\bf x}^m = {\bf h}^\top {\bf x}^{\ast m}$ for an arbitrary vector ${\bf x} \in \mathbb{C}^n$. Assume that $\ten{H}_{qm}$ is a $(qm)^{\rm th}$-order $k$-dimensional Hankel tensor that shares the same generating vector ${\bf h}$ with $\ten{H}_m$. Similarly, it holds that $\ten{H}_{qm} {\bf y}^{qm} = {\bf h}^\top {\bf y}^{\ast qm}$ for an arbitrary vector ${\bf y} \in \mathbb{C}^k$.

If $\ten{H}_m$ is positive semi-definite, then it is equivalent to say that $\ten{H}_m {\bf x}^m = {\bf h}^\top {\bf x}^{\ast m} \geq 0$ for all ${\bf x} \in \mathbb{R}^n$. Note that ${\bf y}^{\ast qm} = ({\bf y}^{\ast q})^{\ast m}$. Thus for an arbitrary vector ${\bf y} \in \mathbb{R}^k$, we have
$$
\ten{H}_{qm} {\bf y}^{qm} = {\bf h}^\top {\bf y}^{\ast qm} = {\bf h}^\top ({\bf y}^{\ast q})^{\ast m} = \ten{H}_m ({\bf y}^{\ast q})^m \geq 0.
$$
Therefore, the higher-order but lower-dimensional Hankel tensor $\ten{H}_{qm}$ is also positive semi-definite. Furthermore, if $\ten{H}_m$ is positive definite, i.e., $\ten{H}_m {\bf x}^m >0$ for all nonzero vector ${\bf x} \in \mathbb{R}^n$, then $\ten{H}_{qm}$ is also positive definite.  We may also derive the negative definite and negative semi-definite case similarly.

If $\ten{H}_m$ is SOS, then there are some multivariate polynomials $p_1,p_2,\dots,p_r$ such that for each ${\bf x} \in \mathbb{R}^n$
$$
\ten{H}_m {\bf x}^m = {\bf h}^\top {\bf x}^{\ast m} = p_1({\bf x})^2 + p_2({\bf x})^2 + \dots + p_r({\bf x})^2.
$$
Thus we also have for each ${\bf y} \in \mathbb{R}^k$
\begin{equation}\label{eq_sos}
\ten{H}_{qm} {\bf y}^{qm} = \ten{H}_m ({\bf y}^{\ast q})^m = p_1({\bf y}^{\ast q})^2 + p_2({\bf y}^{\ast q})^2 + \dots + p_r({\bf y}^{\ast q})^2.
\end{equation}
From the definition of discrete convolutions, we know that ${\bf y}^{\ast q}$ is also a multivariate polynomial about ${\bf y}$. Therefore, the higher-order Hankel tensor $\ten{H}_{qm}$ is also SOS. Moreover, the SOS rank, i.e., the minimum number of squares in the sum-of-squares representations (see \cite{ChenH15}), of $\ten{H}_{qm}$ is no larger than the SOS rank of $\ten{H}_m$. Hence, we summarize the inheritability of positive semi-definiteness and SOS property in the following theorem.

\begin{theorem}\label{thm_inherit}
If an $m^{\rm th}$-order Hankel tensor is positive (negative) (semi-)definite, then the $(qm)^{\rm th}$-order Hankel tensor with the same generating vector is also positive (negative) (semi-)definite for an arbitrary positive integer $q$. If an $m^{\rm th}$-order Hankel tensor is SOS, then the $(qm)^{\rm th}$-order Hankel tensor with the same generating vector is also SOS with no larger SOS rank for an arbitrary positive integer $q$.
\end{theorem}

Let $\ten{T}$ be an $m^{\rm th}$-order $n$-dimensional tensor and ${\bf x} \in \mathbb{C}^n$ be a vector. Recall that $\ten{T}{\bf x}^{m-1}$ is a vector with $(\ten{T}{\bf x}^{m-1})_i = \sum_{i_2,\dots,i_m=1}^n \ten{T}_{i i_2 \dots i_m} x_{i_2} \dots x_{i_m}$. If there is a real scalar $\lambda$ and a nonzero real vector ${\bf x} \in \mathbb{R}^n$ such that $\ten{T}{\bf x}^{m-1} = \lambda {\bf x}^{[m-1]}$, where ${\bf x}^{[m-1]} =: [x_1^{m-1},x_2^{m-1},\dots,x_n^{m-1}]^\top$, then we call $\lambda$ an \emph{H-eigenvalue} of the tensor $\ten{T}$ and ${\bf x}$ a corresponding \emph{H-eigenvector}. This concept was first introduced by Qi \cite{Qi05}, and H-eigenvalues are shown to be very essential for investigating a tensor. By \cite[Theorem 5]{Qi05}, we know that an even order symmetric tensor is positive (semi-)definite if and only if all its H-eigenvalues are positive (nonnegative). Applying the convolution formula, we can further obtain a quantified result about the extremal H-eigenvalues of Hankel tensors.

\begin{theorem}  \label{thm_quant}
  Let $\ten{H}_m$ and $\ten{H}_{qm}$ be two Hankel tensors with the same generating vector and of order $m$ and $qm$, respectively, where $m$ is even. Denote the minimal and the maximal H-eigenvalue of a tensor as $\lambda_{\min}(\cdot)$ and $\lambda_{\max}(\cdot)$, respectively. Then
  $$
  \lambda_{\min}(\ten{H}_{qm}) \geq
  \left\{\begin{array}{ll}
  c_1 \cdot \lambda_{\min}(\ten{H}_{m}), & \text{if } \ten{H}_{qm} \text{ is positive semi-definite}, \\
  c_2 \cdot \lambda_{\min}(\ten{H}_{m}), & \text{otherwise},
  \end{array}\right.
  $$
  and
  $$
  \lambda_{\max}(\ten{H}_{qm}) \leq
  \left\{\begin{array}{ll}
  c_1 \cdot \lambda_{\max}(\ten{H}_{m}), & \text{if } \ten{H}_{qm} \text{ is negative semi-definite}, \\
  c_2 \cdot \lambda_{\max}(\ten{H}_{m}), & \text{otherwise},
  \end{array}\right.
  $$
  where $c_1 = \min_{{\bf y}\in\mathbb{R}^{k}} {\|{\bf y}^{\ast q}\|_{m}^{m}}/{\|{\bf y}\|_{qm}^{qm}}$ and $c_2 = \max_{{\bf y}\in\mathbb{R}^{k}} {\|{\bf y}^{\ast q}\|_{m}^{m}}/{\|{\bf y}\|_{qm}^{qm}}$ are positive constants depending on $m$, $n$, and $q$.
\end{theorem}
\begin{proof}
  Since $\ten{H}_m$ and $\ten{H}_{qm}$ are even order symmetric tensors, from \cite[Theorem 5]{Qi05} we have
  $$
  \lambda_{\min}(\ten{H}_{m})
  = \min_{{\bf x}\in\mathbb{R}^{n}} \frac{\ten{H}_{m} {\bf x}^{m}}{\|{\bf x}\|_{m}^{m}}
  \quad\text{and}\quad
  \lambda_{\min}(\ten{H}_{qm})
  = \min_{{\bf y}\in\mathbb{R}^{k}} \frac{\ten{H}_{qm} {\bf y}^{qm}}{\|{\bf y}\|_{qm}^{qm}},
  $$
  where $n = qk-q+1$.

  If $\ten{H}_{qm}$ is positive semi-definite, i.e., $\ten{H}_{qm} {\bf y}^{qm} \geq 0$ for all ${\bf y}\in\mathbb{R}^{k}$, then we denote $c_1 = \min_{{\bf y}\in\mathbb{R}^{k}} {\|{\bf y}^{\ast q}\|_{m}^{m}}/{\|{\bf y}\|_{qm}^{qm}}$, which is a constant depending only on $m$, $n$, and $q$. Then by the convolution formula proposed above, we have
  $$
  \lambda_{\min}(\ten{H}_{qm})
  \geq c_1 \cdot \min_{{\bf y}\in\mathbb{R}^{k}} \frac{\ten{H}_{qm} {\bf y}^{qm}}{\|{\bf y}^{\ast q}\|_{m}^{m}}
  = c_1 \cdot \min_{{\bf y}\in\mathbb{R}^{k}} \frac{\ten{H}_{m} ({\bf y}^{\ast q})^m}{\|{\bf y}^{\ast q}\|_{m}^{m}}
  \geq c_1 \cdot \min_{{\bf x}\in\mathbb{R}^{n}} \frac{\ten{H}_{m} {\bf x}^{m}}{\|{\bf x}\|_{m}^{m}}
  = c_1 \cdot \lambda_{\min}(\ten{H}_{m}).
  $$

  If $\ten{H}_{qm}$ is not positive semi-definite, then we denote $c_2 = \max_{{\bf y}\in\mathbb{R}^{k}} {\|{\bf y}^{\ast q}\|_{m}^{m}}/{\|{\bf y}\|_{qm}^{qm}}$. Let $\widehat{\bf y}$ be a vector in $\mathbb{R}^k$ such that $\lambda_{\min}(\ten{H}_{qm}) = {\ten{H}_{qm} \widehat{\bf y}^{qm}}/{\|\widehat{\bf y}\|_{qm}^{qm}} < 0$. Then
  $$
  \lambda_{\min}(\ten{H}_{qm})
  \geq c_2 \cdot \frac{\ten{H}_{qm} \widehat{\bf y}^{qm}}{\|\widehat{\bf y}^{\ast q}\|_{m}^{m}}
  = c_2 \cdot \frac{\ten{H}_{m} (\widehat{\bf y}^{\ast q})^m}{\|\widehat{\bf y}^{\ast q}\|_{m}^{m}}
  \geq c_2 \cdot \min_{{\bf x}\in\mathbb{R}^{n}} \frac{\ten{H}_{m} {\bf x}^{m}}{\|{\bf x}\|_{m}^{m}}
  = c_2 \cdot \lambda_{\min}(\ten{H}_{m}).
  $$
  Thus we obtain a lower bound of the minimal H-eigenvalue of $\ten{H}_{qm}$, whenever this tensor is positive semi-definite or not. The proof of the upper bound of the maximal H-eigenvalue of $\ten{H}_{qm}$ is similar.   \qed
\end{proof}

\subsection{SOS decomposition of Strong Hankel Tensors}

When the lower order $m$ in Theorem \ref{thm_inherit} equals $2$, i.e., matrix case,  the $(2q)^{\rm th}$-order Hankel tensor sharing the same generating vector with this positive semi-definite Hankel matrix is called a \emph{strong Hankel tensor}. We shall discuss strong Hankel tensors in detail in later sections. Now we focus on how to write out an SOS decomposition of a strong Hankel tensor following the formula \eqref{eq_sos}. Li et al. \cite{Li14c} showed that even order strong Hankel tensors are SOS tensors. However, their proof is not constructive, and no concrete SOS decomposition is given.

For an arbitrary Hankel matrix $H$ generated by ${\bf h}$, we can compute its \emph{Takagi factorization} efficiently by the algorithm proposed by Qiao et al. \cite{Qiao09a}, where only the generating vector rather than the whole Hankel matrix is required to store. The Takagi factorization can be written into $H = U D U^\top$, where $U = [{\bf u}_1,{\bf u}_2,\dots,{\bf u}_r]$ is a column unitary matrix ($U^\ast U = I$) and $D = {\rm diag}(d_1,d_2,\dots,d_r)$ is a diagonal matrix. When the matrix is real, the Takagi factorization is exactly the singular value decomposition of the Hankel matrix $H$. Further when $H$ is positive semi-definite, the diagonal matrix $D$ is with all nonnegative diagonal entries. Thus the polynomial ${\bf x}^\top H {\bf x}$ can be expressed as a sum of squares $p_1({\bf x})^2 + p_2({\bf x})^2 + \dots + p_r({\bf x})^2$, where
$$
p_k({\bf x}) = d_k^{1/2} {\bf u}_k^\top {\bf x},\quad k = 1,2,\dots,r.
$$

Following the formula \eqref{eq_sos}, the $2q$-degree polynomial $\ten{H}_{2q} {\bf y}^{2q}$ can also be written into a sum of squares $q_1({\bf y})^2 + q_2({\bf y})^2 + \dots + q_r({\bf y})^2$, where
$$
q_k({\bf y}) = d_k^{1/2} {\bf u}_k^\top {\bf y}^{\ast q},\quad k = 1,2,\dots,r.
$$
Recall that any homogenous polynomial is associated with a symmetric tensor. And an interesting observation is that the homogenous polynomial $q_k({\bf y})$ is associated with a $q^{\rm th}$-order Hankel tensor generated by $d_k^{1/2} {\bf u}_k$. Thus we determine an SOS decomposition of a strong Hankel tensor $\ten{H}_{2q}$ by $r$ vectors $d_k^{1/2} {\bf u}_k$ ($k=1,2,\dots,r$). And we summarize the above procedure into the following algorithm.

\begin{algorithm}[htbp]
\caption{ An SOS decomposition of a strong Hankel tensor.}
\label{alg_sos}
\begin{algorithmic}[1]
\Require
The generating vector ${\bf h}$ of a strong Hankel tensor;
\Ensure
An SOS decomposition $q_1({\bf y})^2 + q_2({\bf y})^2 + \dots + q_r({\bf y})^2$ of this Hankel tensor;
\State Compute the Takagi factorization of the Hankel matrix generated by ${\bf h}$: $H = UDU^\top$;
\State ${\bf q}_k = d_k^{1/2} {\bf u}_k$ for $k = 1,2,\dots,r$;
\State {\bf Then} ${\bf q}_k$ generates a $q^{\rm th}$-order Hankel tensor $\ten{Q}_k$ as the coefficient tensor of each term $q_k(\cdot)$ in the SOS decomposition for $k = 1,2,\dots,r$;
\end{algorithmic}
\end{algorithm}

\section{The Second Inheritance Property of Hankel Tensors}

In this section, we prove the second inheritance property of Hankel tensors, i.e., if the associated Hankel matrix of a Hankel tensor has no negative (or non-positive, or positive, or nonnegative) eigenvalues, then that Hankel tensor has no negative (or non-positive, or positive, or nonnegative, respectively) H-eigenvalues.  A basic tool to prove this is an augmented Vandermonde decomposition with positive coefficients.

\subsection{Strong Hankel Tensors}

Let $\ten{H}$ be an $m^{\rm th}$-order $n$-dimensional Hankel tensor generated by the vector ${\bf h} \in \mathbb{R}^{mn-m+1}$. Then the square Hankel matrix $H$ also generated by the vector ${\bf h}$ is called the \emph{associated Hankel matrix} of the Hankel tensor $\ten{H}$. If the associated Hankel matrix of a Hankel tensor $\ten{H}$ is positive semi-definite, then we call this Hankel tensor $\ten{H}$ a \emph{strong Hankel tensor} (see \cite{Qi15}). When $m$ is even, we immediately know that an even order strong Hankel tensor must be positive semi-definite by Theorem \ref{thm_inherit}, which has been proved in \cite[Theorem 3.1]{Qi15}.

Qi \cite{Qi15} also introduced the \emph{Vandermonde decomposition} of a Hankel tensor
\begin{equation}\label{eq_VD}
\ten{H} = \sum_{k=1}^r \alpha_k {\bf v}_k^{\circ m},
\end{equation}
where ${\bf v}_k$ is in the Vandermonde form $\big[1,\xi_k,\xi_k^2,\dots,\xi_k^{n-1}\big]^\top$, ${\bf v}^{\circ m} :=\underbrace{ {\bf v} \circ {\bf v} \circ \dots \circ {\bf v}}_m$ is a rank-one tensor, and the outer product is defined by
$$
({\bf v}_1 \circ {\bf v}_2 \circ \dots \circ {\bf v}_m)_{i_1 i_2 \dots i_m} = ({\bf v}_1)_{i_1} ({\bf v}_2)_{i_2} \cdots ({\bf v}_m)_{i_m}.
$$
The Vandermonde decomposition is equivalent to factorize the generating vector of $\ten{H}$, i.e.,
$$
\begin{bmatrix}
h_0 \\ h_1 \\ h_2 \\ \vdots \\ h_r
\end{bmatrix}
= \begin{bmatrix}
1 & 1 & \cdots & 1 \\
\xi_1 & \xi_2 & \cdots & \xi_r \\
\xi_1^2 & \xi_2^2 & \cdots & \xi_r^2 \\
\vdots & \vdots & \vdots & \vdots \\
\xi_1^{mn-m} & \xi_2^{mn-m} & \cdots & \xi_r^{mn-m}
\end{bmatrix}
\cdot
\begin{bmatrix}
\alpha_1 \\ \alpha_2 \\ \vdots \\ \alpha_r
\end{bmatrix}.
$$
Since the above Vandermonde matrix is nonsingular if and only if $r = mn-m+1$ and $\xi_1,\xi_2,\dots,\xi_r$ are mutually distinct, every Hankel tensor must have such a Vandermonde decomposition with the number of items no larger than $mn-m+1$. When the Hankel tensor is further positive semi-definite, we desire that all the coefficients $\alpha_k$'s are positive, so that each item is a rank-one positive semi-definite Hankel tensor when $m$ is even. Moreover, a real square Hankel tensor $\ten{H}$ is called a \emph{complete Hankel tensor}, if the coefficients $\alpha_k$'s in one of its Vandermonde decompositions are all positive (see \cite{Qi15}).

Nevertheless, the set of all complete Hankel tensors is not ``complete''. Li et al. showed in \cite[Corollary 1]{Li14c} that the $m^{\rm th}$-order $n$-dimensional complete Hankel tensor cone is not closed and its closure is the $m^{\rm th}$-order $n$-dimensional strong Hankel tensor cone. An obvious counterexample is ${\bf e}_n^{\circ m}$, where ${\bf e}_n = [0,\dots,0,1]^\top$ is the last unit vector. Since all the Vandermonde vectors begin with a $1$ and $\alpha_k$'s being positive, the positive semi-definite Hankel tensor ${\bf e}_n^{\circ m}$ must not be a complete Hankel tensor.

Fortunately, $\alpha{\bf e}_n^{\circ m}$ is the only kind of rank-one non-complete Hankel tensors. We will shortly show that if we add it into the basis, then all the strong Hankel tensors can be decomposed into an \emph{augmented Vandermonde decomposition}
$$
\ten{H} = \sum_{k=1}^{r-1} \alpha_k {\bf v}_k^{\circ m} + \alpha_r {\bf e}_n^{\circ m}.
$$
Note that $\frac{1}{\xi^{n-1}}[1,\xi,\xi^2,\dots,\xi^{n-1}]^\top \to {\bf e}_n$ when $\xi \to \infty$. The cone of Hankel tensors with an augmented Vandermonde decomposition is actually the closure of the cone of complete Hankel tensors.
When a Hankel tensor $\ten{H}$ has such an augmented Vandermonde decomposition, its associated Hankel matrix $H$ also has a corresponding decomposition
$$
H = \sum_{k=1}^{r-1} \alpha_k \widetilde{\bf v}_k^{\circ 2} + \alpha_r {\bf e}_{(n-1)m/2+1}^{\circ 2},
$$
where $\widetilde{\bf v}_k = \big[1,\xi_k,\xi_k^2,\dots,\xi_k^{(n-1)m/2}\big]^\top$ and ${\bf v}^{\circ 2}$ is exactly ${\bf v} {\bf v}^\top$, and vice versa. Therefore, if a positive semi-definite Hankel tensor has an augmented Vandermonde decomposition with all positive coefficients, then it is also a strong Hankel tensor, that is, its associated Hankel matrix must be positive semi-definite. Furthermore, when we obtain an augmented Vandermonde decomposition of its associated Hankel matrix, we can induce an augmented Vandermonde decomposition of the original Hankel tensor straightforward. Hence, we begin with the positive semi-definite Hankel matrices.

\subsection{A General Vandermonde Decompostion}

We shall introduce the algorithm for a general Vandermonde decomposition of an arbitrary Hankel matrix proposed by Boley et al. \cite{Boley97} briefly in this subsection first, then restrict the decomposition into the real situation. Let's begin with a nonsingular Hankel matrix $H \in \mathbb{C}^{r \times r}$. After we solve the Yule-Walker equation \cite[Chapter 4.7]{Golub13}
$$
\begin{bmatrix}
h_0 & h_1 & h_2 & \cdots & h_{r-1} \\
h_1 & h_2 & h_3 & \cdots & h_r \\
\vdots & \vdots & \vdots & \iddots & \vdots \\
h_{r-2} & h_{r-1} & h_r & \cdots & h_{2r-3} \\
h_{r-1} & h_r & h_{r+1} & \cdots & h_{2r-2}
\end{bmatrix}
\cdot
\begin{bmatrix}
a_0 \\ a_1 \\ a_2 \\ \vdots \\ a_{r-1}
\end{bmatrix}
=
\begin{bmatrix}
h_r \\ h_{r+1} \\ \vdots \\ h_{2r-2} \\ \gamma
\end{bmatrix},
$$
we obtain an $r$ term recurrence for $k = r,r+1,\dots,2r-2$, i.e.,
$$
h_k = a_{r-1} h_{k-1} + a_{r-2} h_{k-2} + \dots + a_0 h_{k-r}.
$$
Denote $C$ as the companion matrix \cite[Chapter 7.4.6]{Golub13} corresponding to the polynomial $p(\lambda) = \lambda^r - a_{r-1} \lambda^{r-1} - \dots - a_0 \lambda^0$, i.e.,
$$
C = \begin{bmatrix}
0 & 1 \\
  & 0 & 1 \\
  &   & \ddots & \ddots \\
  &   &        & 0 & 1 \\
a_0 & a_1 & \cdots & a_{r-2} & a_{r-1}
\end{bmatrix}.
$$
Let the Jordan canonical form of $C$ be $C = V^\top J V^{-\top}$, where $J = {\rm diag}\{J_1,J_2,\dots,J_s\}$ and $J_l$ is the $k_l \times k_l$ Jordan block corresponding to eigenvalue $\lambda_l$. Moreover, the nonsingular matrix $V$ has the form
$$
V = \big[{\bf v}, J^\top {\bf v}, (J^\top)^2 {\bf v}, \dots, (J^\top)^{r-1} {\bf v}\big],
$$
where ${\bf v} = [{\bf e}_{k_1,1}^\top,{\bf e}_{k_2,1}^\top,\dots,{\bf e}_{k_s,1}^\top]^\top$ is a vector partitioned conformably with $J$ and ${\bf e}_{k_l,1}$ is the first $k_l$-dimensional unit coordinate vector. This kind of $V$ is often called a \emph{confluent Vandermonde matrix}. When the multiplicities of all the eigenvalues of $C$ equal one, the matrix $V$ is exactly a Vandermonde matrix.

Denote ${\bf h}_0$ as the first column of $H$ and ${\bf w} = V^{-\top} {\bf h}_0$. There exists a unique block diagonal matrix $D = {\rm diag}\{D_1,D_2,\dots,D_s\}$, which is also partitioned conformably with $J$, satisfying
$$
D {\bf v} = {\bf w} \quad\text{and}\quad D J^\top = J D.
$$
Moreover, each block $D_l$ is a $k_l$-by-$k_l$ upper anti-triangular Hankel matrix. If we partition ${\bf w} = [{\bf w}_1,{\bf w}_2,\dots,{\bf w}_s]^\top$ conformably with $J$, then the $l^{\rm th}$ block is determined by
$$
D_l = \begin{bmatrix}
({\bf w}_l)_1 & ({\bf w}_l)_2 & \cdots & ({\bf w}_l)_{k_l} \\
({\bf w}_l)_2 & \cdots & ({\bf w}_l)_{k_l} & 0 \\
\vdots & \iddots & \iddots & \vdots \\
({\bf w}_l)_{k_l} & 0 & \cdots & 0
\end{bmatrix}.
$$
Finally, we obtain a general Vandermonde decomposition of a full-rank Hankel matrix
$$
H = V^\top D V.
$$
Furthermore, if the leading $r \times r$ principal submatrix, i.e., $H(1:r,1:r)$, of an $n \times n$ rank-$r$ Hankel matrix $H$ is nonsingular, then $H$ admits the Vandermonde decomposition $H = (V_{r \times n})^\top D_{r \times r} V_{r \times n}$, which is induced by the decomposition of the leading $r \times r$ principal submatrix.

Nevertheless, this generalized Vandermonde decomposition is not enough for discussing the positive definiteness of a real Hankel matrix, since the factors $V$ and $D$ could be complex even though $H$ is a real matrix. We shall modify this decomposition into a general real Vandermonde decomposition. Assume that two eigenvalues $\lambda_1$ and $\lambda_2$ of $C$ form a pair of conjugate complex numbers. Then the corresponding parts in $D$ and $V$ are also conjugate, respectively. That is,
$$
\begin{bmatrix} V_1^\top & V_2^\top \end{bmatrix} \cdot \begin{bmatrix} D_1 & \\ & D_2 \end{bmatrix} \cdot \begin{bmatrix} V_1 \\ V_2 \end{bmatrix} =
\begin{bmatrix} V_1^\top & \bar{V}_1^\top \end{bmatrix} \cdot \begin{bmatrix} D_1 & \\ & \bar{D}_1 \end{bmatrix} \cdot \begin{bmatrix} V_1 \\ \bar{V}_1 \end{bmatrix}.
$$
Note that
$$
\begin{bmatrix} {\bf u} + {\bf v}\imath & {\bf u} - {\bf v}\imath \end{bmatrix} \cdot \begin{bmatrix} a + b\imath & \\ & a - b\imath \end{bmatrix} \cdot \begin{bmatrix} {\bf u}^\top + {\bf v}^\top\imath \\ {\bf u}^\top - {\bf v}^\top\imath \end{bmatrix} =
\begin{bmatrix} {\bf u} & {\bf v} \end{bmatrix} \cdot 2\begin{bmatrix} a & -b \\ -b & -a \end{bmatrix} \cdot \begin{bmatrix} {\bf u}^\top \\ {\bf v}^\top \end{bmatrix}.
$$
Denote the $j^{\rm th}$ column of $V_1^\top$ as ${\bf u}_j + {\bf v}_j \imath$ and the $j^{\rm th}$ entry of the first column of $D_1$ is $a_j + b_j \imath$, where $\imath = \sqrt{-1}$ and ${\bf u}_j, {\bf v}_j, a_j, b_j$ are all real. Then
\[
\begin{split}
&\begin{bmatrix} V_1^\top & V_2^\top \end{bmatrix} \cdot \begin{bmatrix} D_1 & \\ & D_2 \end{bmatrix} \cdot \begin{bmatrix} V_1 \\ V_2 \end{bmatrix} \\
&=
\begin{bmatrix} {\bf u}_1 & {\bf v}_1 & \dots & {\bf u}_{k_1} & {\bf v}_{k_1} \end{bmatrix} \cdot
\begin{bmatrix}
\Lambda_1 & \Lambda_2 & \cdots & \Lambda_{k_1} \\
\Lambda_2 & \cdots & \Lambda_{k_1} & {\bf O} \\
\vdots & \iddots & \iddots & \vdots \\
\Lambda_{k_l} & {\bf O} & \cdots & {\bf O}
\end{bmatrix} \cdot
\begin{bmatrix} {\bf u}_1^\top \\ {\bf v}_1^\top \\ \vdots \\ {\bf u}_{k_1}^\top \\ {\bf v}_{k_1}^\top \end{bmatrix},
\end{split}
\]
where the $2$-by-$2$ block $\Lambda_j$ is
$$
\Lambda_j = 2\begin{bmatrix} a_j & -b_j \\ -b_j & -a_j \end{bmatrix}.
$$
We do the same transformations to all the conjugate eigenvalue pairs, then we obtain a real decomposition of the real Hankel matrix $H = \widehat{V}^\top \widehat{D} \widehat{V}$. Here, each diagonal block of $\widehat{D}$ corresponding to a real eigenvalue of $C$ is an upper anti-triangular Hankel matrix, and each one corresponding to a pair of two conjugate eigenvalues is an upper anti-triangular block Hankel matrix with $2$-by-$2$ blocks.

We claim that if the Hankel matrix $H$ is further positive semi-definite, then all the eigenvalues of $C$ are real and of multiplicity one. This can be seen by recognizing that the following three cases of the diagonal blocks of $\widehat{D}$ cannot be positive semi-definite: (1) an anti-upper triangular Hankel block whose size is larger than $1$, (2) a $2$-by-$2$ block $\Lambda_j = 2\left[\begin{smallmatrix} a_j & -b_j \\ -b_j & -a_j \end{smallmatrix}\right]$, and (3) a block anti-upper triangular Hankel block with the blocks in case (2). Therefore, when a real rank-$r$ Hankel matrix $H$ is positive semi-definite and its leading $r \times r$ principal submatrix is positive definite, the block diagonal matrix $\widehat{D}$ in the generalized real Vandermonde decomposition must be diagonal. Hence this Hankel matrix admits a Vandermonde decomposition with $r$ terms and all positive coefficients
$$
H = \sum_{k=1}^r \alpha_k {\bf v}_k {\bf v}_k^\top, \quad \alpha_k > 0, \quad {\bf v}_k = \big[1,\xi_k,\dots,\xi_k^{n-1}\big]^\top.
$$
This result for positive definite Hankel matrices is already known (see \cite[Lemma 0.2.1]{Tyrtyshnikov94}).

\subsection{The Second Inheritance Property of Hankel Tensors}

However, the associated Hankel matrix of a Hankel tensor is not necessarily with a full-rank leading principal submatrix. Thus we shall study whether a positive semi-definite Hankel matrix can always decomposed into the form
$$
H = \sum_{k=1}^{r-1} \alpha_k {\bf v}_k {\bf v}_k^\top + \alpha_r {\bf e}_n {\bf e}_n^\top, \quad \alpha_k \geq 0.
$$
We first need a lemma about the rank-one modifications on a positive semi-definite matrix. Denote the range and the kernel of a matrix $A$ as ${\rm Ran}(A)$ and ${\rm Ker}(A)$, respectively.

\begin{lemma}\label{lemma_rank1}
Let $A$ be a positive semi-definite matrix with ${\rm rank}(A) = r$. Then there exists a unique $\alpha>0$ such that $A - \alpha {\bf u} {\bf u}^\top$ is positive semi-definite with ${\rm rank}(A - \alpha {\bf u} {\bf u}^\top) = r-1$, if and only if ${\bf u}$ is in the range of $A$.
\end{lemma}

\begin{proof}
The condition ${\rm rank}(A - \alpha {\bf u} {\bf u}^\top) = {\rm rank}(A)-1$ obviously indicates that ${\bf u} \in {\rm Ran}(A)$. Thus we only need to prove the ``if'' part of the statement.

Let the nonzero eigenvalues of $A$ be $\lambda_1,\lambda_2,\dots,\lambda_r$ and the corresponding eigenvectors be ${\bf x}_1,{\bf x}_2,\dots,{\bf x}_r$, respectively. Since ${\bf u} \in {\rm Ran}(A)$, we can write ${\bf u} = \mu_1 {\bf x}_1 + \mu_2 {\bf x}_2 + \dots + \mu_r {\bf x}_r$. Note that ${\rm rank}(A - \alpha {\bf u} {\bf u}^\top) = {\rm rank}(A)-1$ also implies ${\rm dim}\,{\rm Ker}(A - \alpha {\bf u} {\bf u}^\top) = {\rm dim}\,{\rm Ker}(A)+1$, which is equivalent to that there exists a unique subspace ${\rm span}\{\bf p\}$ such that $A {\bf p} = \alpha {\bf u} ({\bf u}^\top {\bf p}) \neq {\bf 0}$. Write ${\bf p} = \eta_1 {\bf x}_1 + \eta_2 {\bf x}_2 + \dots + \eta_r {\bf x}_r$. Then there exists a unique linear combination and a unique scalar $\alpha$ satisfying the requirement, i.e.,
$$
\eta_i = \mu_i/\lambda_i\ (i=1,2,\dots,r) \quad\text{and}\quad \alpha = \big(\mu_1^2/\lambda_1 + \dots + \mu_r^2/\lambda_r\big)^{-1}.
$$

Then we need to verify the positive semi-definiteness of $A - \alpha {\bf u} {\bf u}^\top$. For an arbitrary vector ${\bf x} = \xi_1 {\bf x}_1 + \xi_2 {\bf x}_2 + \dots + \xi_r {\bf x}_r$ in the range of $A$, we have
$$
{\bf x}^\top A {\bf x} = \xi_1^2 \lambda_1 + \dots + \xi_r^2 \lambda_r
\quad\text{and}\quad
{\bf u}^\top {\bf x} = \mu_1 \xi_1 + \dots + \mu_r \xi_r.
$$
Then, along with the expression of $\alpha$, the H\"{o}lder inequality indicates that ${\bf x}^\top A {\bf x} \geq \alpha ({\bf u}^\top {\bf x})^2$, i.e., the rank-$(r-1)$ matrix $A - \alpha {\bf u} {\bf u}^\top$ is also positive semi-definite.  \qed
\end{proof}

The following theorem tells that the leading $(r-1) \times (r-1)$ principal submatrix of a rank-$r$ positive semi-definite Hankel matrix is always full-rank, even when the leading $r \times r$ principal submatrix is rank deficient.

\begin{theorem}
Let $H \in \mathbb{R}^{n \times n}$ be a positive semi-definite Hankel matrix with ${\rm rank}(H) = r$. If the last column $H(:,n)$ is linearly dependent with the first $n-1$ columns $H(:,1:n-1)$, then the leading $r \times r$ principal submatrix $H(1:r,1:r)$ is positive definite. If $H(:,n)$ is linearly independent with $H(:,1:n-1)$, then the leading $(r-1) \times (r-1)$ principal submatrix $H(1:r-1,1:r-1)$ is positive definite.
\end{theorem}

\begin{proof}
We apply the mathematical induction on the size $n$. First, the statement is apparently true for $2 \times 2$ positive semi-definite Hankel matrices. Assume that the statement holds for $(n-1) \times (n-1)$ Hankel matrices, then we consider the $n \times n$ case.

Case 1: When the last column $H(:,n)$ is linearly dependent with the first $n-1$ columns $H(:,1:n-1)$, the submatrix $H(1:n-1,1:n-1)$ is also a rank-$r$ positive semi-definite Hankel matrix. Then from the induction hypothesis, $H(1:r,1:r)$ is full rank if $H(1:n-1,n-1)$ is linearly dependent with $H(1:n-1,1:n-2)$, and $H(1:r-1,1:r-1)$ is full rank otherwise. We shall show that the column $H(1:n-1,n-1)$ is always linearly dependent with $H(1:n-1,1:n-2)$.

Otherwise, the leading $(r-1) \times (r-1)$ principal submatrix $H(1:r-1,1:r-1)$ is positive definite, and the rank of $H(1:n-2,1:n-1)$ is $r-1$. Since the column $H(:,n)$ is linear dependent with the previous $(n-1)$ columns, the rank of $H(1:n-2,:)$ is also $r-1$. Thus the rectangular Hankel matrix $H(1:n-2,:)$ has a Vandermonde decomposition
$$
H(1:n-2,:) = \sum_{k=1}^{r-1} \alpha_k \begin{bmatrix} 1 \\ \xi_k \\ \vdots \\ \xi_k^{n-3} \end{bmatrix} \begin{bmatrix} 1 & \xi_k & \cdots & \xi_k^{n-2} & \xi_k^{n-1} \end{bmatrix}.
$$
Since $H(n-1,n-1) = H(n-2,n)$, the square Hankel matrix $H(1:n-1,1:n-1)$ has a corresponding decomposition
$$
H(1:n-1,1:n-1) = \sum_{k=1}^{r-1} \alpha_k \begin{bmatrix} 1 \\ \xi_k \\ \vdots \\ \xi_k^{n-2} \end{bmatrix} \begin{bmatrix} 1 & \xi_k & \cdots & \xi_k^{n-2} \end{bmatrix}.
$$
This contradicts with ${\rm rank}\big(H(1:n-1,1:n-1)\big) = r$. Therefore, $H(1:n-1,n-1)$ must be linearly dependent with $H(1:n-1,1:n-2)$. Hence, the leading principal submatrix $H(1:r,1:r)$ is positive definite.

Case 2: When the last column $H(:,n)$ is linearly independent with the first $n-1$ columns $H(:,1:n-1)$, it is equivalent to that ${\bf e}_n$ is in the range of $H$. Thus, from Lemma \ref{lemma_rank1}, there exists a scalar $\alpha_r$ such that $H - \alpha_r {\bf e}_n {\bf e}_n^\top$ is rank-$(r-1)$ and also positive semi-definite. Referring back to Case 1, we know that the leading principal submatrix $H(1:r-1,1:r-1)$ is positive definite.    \qed
\end{proof}

Following the above theorem, when $H(:,n)$ is linearly dependent with $H(:,1:n-1)$, the leading $r \times r$ principal submatrix $H(1:r,1:r)$ is positive definite. Thus $H$ has a Vandermonde decomposition with all positive coefficients
$$
H = \sum_{k=1}^r \alpha_k {\bf v}_k {\bf v}_k^\top, \quad \alpha_k > 0, \quad {\bf v}_k = \big[1,\xi_k,\dots,\xi_k^{n-1}\big]^\top.
$$
When $H(:,n)$ is linearly independent with $H(:,1:n-1)$, the leading $(r-1) \times (r-1)$ principal submatrix $H(1:r-1,1:r-1)$ is positive definite. Thus $H$ has an augmented Vandermonde decomposition with all positive coefficients
$$
H = \sum_{k=1}^{r-1} \alpha_k {\bf v}_k {\bf v}_k^\top + \alpha_r {\bf e}_n {\bf e}_n^\top, \quad \alpha_k > 0, \quad {\bf v}_k = \big[1,\xi_k,\dots,\xi_k^{n-1}\big]^\top.
$$
Combining the definition of strong Hankel tensors and the analysis at the end of Section 3.1, we arrive at the following theorem.

\begin{theorem}\label{thm_strong}
Let $\ten{H}$ be an $m^{\rm th}$-order $n$-dimensional Hankel tensor and the rank of its associated Hankel matrix be $r$. Then it is a strong Hankel tensor if and only if it admits a Vandermonde decomposition with all positive coefficients
\begin{equation}
\ten{H} = \sum_{k=1}^r \alpha_k {\bf v}_k^{\circ m}, \quad \alpha_k > 0, \quad {\bf v}_k = \big[1,\xi_k,\dots,\xi_k^{n-1}\big]^\top,
\end{equation}
or an augmented Vandermonde decomposition with all positive coefficients
\begin{equation}\label{eq_AVD}
\ten{H} = \sum_{k=1}^{r-1} \alpha_k {\bf v}_k^{\circ m} + \alpha_r {\bf e}_n^{\circ m}, \quad \alpha_k > 0, \quad {\bf v}_k = \big[1,\xi_k,\dots,\xi_k^{n-1}\big]^\top.
\end{equation}
\end{theorem}

After we obtain the above theorem, the strong Hankel tensor cone is actually understood thoroughly. The polynomials induced by strong Hankel tensors are not only positive semi-definite and sum-of-squares, as proved in \cite[Theorem 3.1]{Qi15} and \cite[Corollary 2]{Li14c}, but also sum-of-$m^{\rm th}$-powers indeed. The detailed algorithm for computing an augmented Vandermonde decomposition of a strong Hankel tensor is displayed as follows.

\begin{algorithm}[htbp]
\caption{ Augmented Vandermonde decomposition of a strong Hankel tensor.}
\label{alg_vander}
\begin{algorithmic}[1]
\Require
The generating vector ${\bf h}$ of a strong Hankel tensor;
\Ensure
Coefficients $\alpha_k$; Poles $\xi_k$;
\State Compute the Takagi factorization of the Hankel matrix $H$ generated by ${\bf h}$: $H = U D U^\top$;
\State Recognize the rank $r$ of $H$ and whether ${\bf e}_n$ is in the range of $U$;
\State {\bf If} ${\bf e}_n \notin {\rm Ran}(U)$, {\bf then}
\State \hspace{12pt} {\bf If} $r < n$, {\bf then}
\State \hspace{12pt}\hspace{12pt} ${\bf a} = U(1:r,1:r)^{-\top} D(1:r,1:r)^{-1} U(1:r,1:r)^{-1} {\bf h}(r:2r-1)$;
\State \hspace{12pt} {\bf ElseIf} $r = n$, {\bf then}
\State \hspace{12pt}\hspace{12pt} ${\bf a} = U(1:r,1:r)^{-\top} D(1:r,1:r)^{-1} U(1:r,1:r)^{-1} [{\bf h}(r:2r-2)^\top,\gamma]^\top$, where $\gamma$ is arbitrary;
\State \hspace{12pt} {\bf EndIf}
\State \hspace{12pt} Compute the roots $\xi_1,\xi_2,\dots,\xi_r$ of the polynomial $p(\xi) = \xi^r - a_{r-1} \xi^{r-1} - \dots - a_0 \xi^0$;
\State \hspace{12pt} Solve the Vandermonde system ${\rm vander}(\xi_1,\xi_2,\dots,\xi_r) \cdot [\alpha_1,\alpha_2,\dots,\alpha_r]^\top = {\bf h}(0:r-1)$;
\State {\bf Else}
\State \hspace{12pt} $\xi_r = {\rm Inf}$;
\State \hspace{12pt} $\alpha_r  = \big(\sum_{j=1}^r U(n,j)^2 / D(j,j)\big)^{-1}$;
\State \hspace{12pt} Apply Algorithm \ref{alg_vander} for the strong Hankel tensor generated by $[h_0,\dots,h_{mn-m-1},h_{mn-m}-\alpha_r]$ to compute $\alpha_k$ and $\xi_k$ for $k=1,2,\dots,r-1$;
\State {\bf EndIf} \\
\Return $\alpha_k$, $\xi_k$ for $k=1,2,\dots,r$;
\end{algorithmic}
\end{algorithm}

Employing the augmented Vandermonde decomposition, we can prove the following theorem.

\begin{theorem}
  If a Hankel matrix has no negative (or non-positive, or positive, or nonnegative) eigenvalues, then all of its associated higher-order Hankel tensors have no negative (or non-positive, or positive, or nonnegative, respectively) H-eigenvalues.
\end{theorem}
\begin{proof}
  This statement for even order case is a direct corollary of Theorem \ref{thm_inherit}.

  Suppose that the Hankel matrix has no negative eigenvalues.   When the order $m$ is odd, decompose an $m^{\rm th}$-order strong Hankel tensor $\ten{H}$ into $\ten{H} = \sum_{k=1}^{r-1} \alpha_k {\bf v}_k^{\circ m} + \alpha_r {\bf e}_n^{\circ m}$ with $\alpha_k \geq 0$ ($k=1,2,\dots,r$). Then for an arbitrary vector ${\bf x}$, the first entry of $\ten{H}{\bf x}^{m-1}$ is
  $$
  (\ten{H}{\bf x}^{m-1})_1 = \sum_{k=1}^{r-1} ({\bf v}_k)_1 \cdot \alpha_k ({\bf v}_k^\top {\bf x})^{m-1} = \sum_{k=1}^{r-1} \alpha_k ({\bf v}_k^\top {\bf x})^{m-1} \geq 0.
  $$
  If $\ten{H}$ has no H-eigenvalues, then it is done.   Assume it has at least one H-eigenvalue.
  Let $\lambda$ be an H-eigenvalue of $\ten{H}$ and ${\bf x}$ be a corresponding H-eigenvector. Then when $x_1 \neq 0$, from the definition we have
  $$
  \lambda = (\ten{H}{\bf x}^{m-1})_1/x_1^{m-1} \geq 0,
  $$
  since $m$ is an odd number. When $x_1 = 0$, we know that $(\ten{H}{\bf x}^{m-1})_1$ must also be zero, thus all the item $\alpha_k ({\bf v}_k^\top {\bf x})^{m-1} = 0$ for $k=1,2,\dots,r-1$. So the tensor-vector product
  $$
  \ten{H} {\bf x}^{m-1} = \sum_{k=1}^{r-1} {\bf v}_k \cdot \alpha_k ({\bf v}_k^\top {\bf x})^{m-1} + {\bf e}_n \cdot \alpha_r x_n^{m-1} = {\bf e}_n \cdot \alpha_r x_n^{m-1},
  $$
  and it is apparent that ${\bf x} = {\bf e}_n$ and $\lambda = \alpha_r x_n^{m-1}$, where the H-eigenvalue $\lambda$ is also nonnegative.    The other cases can be proved similarly.   \qed
\end{proof}

When the Hankel matrix has no negative eigenvalue, it is positive semi-definite, i.e., the associated Hankel tensors are strong Hankel tensors, which may be of either even or odd order.   Thus, we have the following corollary.

\begin{corollary}
Strong Hankel tensors have no negative H-eigenvalues.
\end{corollary}

We also have a quantified version of the second inheritance property.

\begin{theorem}
  Let $\ten{H}$ be an $m^{\rm th}$-order $n$-dimensional Hankel tensor, and $H$ be its associated Hankel matrix. If $H$ is positive semi-definite, then
  $$
  \lambda_{\min}(\ten{H}) \geq c \cdot \lambda_{\min}(H),
  $$
  where $c$ is a positive constant depending on $m$ and $n$, $c = \min_{{\bf y}\in\mathbb{R}^{n}} {\|{\bf y}^{\ast \frac{m}{2}}\|_{2}^{2}}/{\|{\bf y}\|_{m}^{m}}$ if $m$ is even, and
  $c = \min_{{\bf y}\in\mathbb{R}^{n}} {\|{\bf y}^{\ast \frac{m-1}{2}}\|_2^2}/{\|{\bf y}\|_{m-1}^{m-1}}$ if $m$ is odd.
  If $H$ is negative semi-definite, then
  $$
  \lambda_{\max}(\ten{H}) \geq c \cdot \lambda_{\max}(H).
  $$
\end{theorem}
\begin{proof}
  When the minimal eigenvalue of $H$ equals $0$, the above equality holds for every nonnegative $c$. Moreover, when the order $m$ is even, Theorem \ref{thm_quant} gives the constant $c$. Thus, we need only to discuss the situation that $H$ is positive definite and $m$ is odd.

  Since $H$ is positive definite, the Hankel tensor $\ten{H}$ has a standard Vandermonde decomposition with positive coefficients
  $$
  \ten{H} = \sum_{k=1}^r \alpha_k {\bf v}_k^{\circ m}, \quad \alpha_k > 0, \quad {\bf v}_k = \big[1,\xi_k,\dots,\xi_k^{n-1}\big]^\top,
  $$
  where $\xi_k$'s are mutually distinct. Then by proof of the above theorem, for an arbitrary nonzero vector ${\bf x} \in \mathbb{R}^n$
  $$
  (\ten{H}{\bf x}^{m-1})_1 = \sum_{k=1}^r \alpha_k ({\bf v}_k^\top {\bf x})^{m-1} > 0,
  $$
  since $[{\bf v}_1,{\bf v}_2,\dots,{\bf v}_r]$ spans the whole space. So if $\lambda$ and ${\bf x}$ be an H-eigenvalue and a corresponding H-eigenvector of $\ten{H}$, then $\lambda$ must be positive and the first entry $x_1$ of ${\bf x}$ must be nonzero.

  Let ${\bf h} \in \mathbb{R}^{m(n-1)+1}$ be the generating vector of both the Hankel matrix $H$ and the Hankel tensor $\ten{H}$. Denote ${\bf h}_1 = {\bf h}\big(1:(m-1)(n-1)+1\big)$, which generates a Hankel matrix $H_1$ and an $(m-1)^{\rm st}$-order Hankel tensor $\ten{H}_1$. Note that $H_1$ is a leading principal submatrix of $H$ and $\ten{H}_1$ is exactly the first row tensor of $\ten{H}$, i.e., $\ten{H}(1,:,:,\dots,:)$. Then we have
  $$
  \lambda = \frac{(\ten{H}{\bf x}^{m-1})_1}{x_1^{m-1}} = \frac{\ten{H}_1{\bf x}^{m-1}}{x_1^{m-1}} \geq \frac{\ten{H}_1{\bf x}^{m-1}}{\|{\bf x}\|_{m-1}^{m-1}}.
  $$
  Now $m-1$ is an even number, then we know from Theorem \ref{thm_quant} that there is a constant $c$ such that $\lambda_{\min}(\ten{H}_1) \geq c \cdot \lambda_{\min}(H_1)$. Therefore, we finally obtain that for each H-eigenvalue $\lambda$ of $\ten{H}$, if exists,
  $$
  \lambda \geq \lambda_{\min}(\ten{H}_1) \geq c \cdot \lambda_{\min}(H_1) \geq c \cdot \lambda_{\min}(H).
  $$
  The last inequality holds because $H_1$ is a principal submatrix of $H$  (see \cite[Theorem 8.1.7]{Golub13}).   \qed
\end{proof}

It is not clear whether we have a similar quantified form of the extremal H-eigenvalues of a Hankel tensor when its associated Hankel matrix has both positive and negative eigenvalues.

\section{Numerical Examples}

We shall employ some numerical examples to illustrate the proposed algorithms and theory.

\begin{example}
The first simple example is a $4^{\rm th}$-order $3$-dimensional Hankel tensor $\ten{H}$ generated by $[1,0,1,0,1,0,1,0,1]^\top \in \mathbb{R}^9$. The Takagi factorization of the Hankel matrix generated by the same vector is
$$
\begin{bmatrix}
1 & 0 & 1 & 0 & 1 \\
0 & 1 & 0 & 1 & 0 \\
1 & 0 & 1 & 0 & 1 \\
0 & 1 & 0 & 1 & 0 \\
1 & 0 & 1 & 0 & 1
\end{bmatrix} =
\begin{bmatrix}
\frac{1}{\sqrt{3}} & 0 \\
0 & \frac{1}{\sqrt{2}} \\
\frac{1}{\sqrt{3}} & 0 \\
0 & \frac{1}{\sqrt{2}} \\
\frac{1}{\sqrt{3}} & 0
\end{bmatrix} \cdot
\begin{bmatrix}
3 & 0 \\
0 & 2
\end{bmatrix} \cdot
\begin{bmatrix}
\frac{1}{\sqrt{3}} & 0 & \frac{1}{\sqrt{3}} & 0 & \frac{1}{\sqrt{3}} \\
0 & \frac{1}{\sqrt{2}} & 0 & \frac{1}{\sqrt{2}} & 0
\end{bmatrix}.
$$
Thus by Algorithm \ref{alg_sos}, an SOS decomposition of $\ten{H}{\bf y}^4$ is obtained, that is,
\[
\begin{split}
&\Bigg(\begin{bmatrix} y_1 & y_2 & y_3 \end{bmatrix} \cdot \begin{bmatrix} 1 & 0 & 1 \\ 0 & 1 & 0 \\ 1 & 0 & 1 \end{bmatrix} \cdot \begin{bmatrix} y_1 \\ y_2 \\ y_3 \end{bmatrix}\Bigg)^2 + \Bigg(\begin{bmatrix} y_1 & y_2 & y_3 \end{bmatrix} \cdot \begin{bmatrix} 0 & 1 & 0 \\ 1 & 0 & 1 \\ 0 & 1 & 0 \end{bmatrix} \cdot \begin{bmatrix} y_1 \\ y_2 \\ y_3 \end{bmatrix}\Bigg)^2 \\
&= (y_1^2 + y_2^2 + y_3^2 + 2y_1 y_3)^2 + (2y_1 y_2 + 2y_2 y_3)^2.
\end{split}
\]
However, the SOS decomposition is not unique, as we have seen in the last section that $\ten{H}{\bf y}^4$ can also be written into $\frac{1}{2}(y_1 + y_2 + y_3)^4 + \frac{1}{2}(y_1 - y_2 + y_3)^4$.
\end{example}

An $m^{\rm th}$-order $n$-dimensional Hilbert tensor (see \cite{Song14}) is defined by
$$
\ten{H}(i_1,i_2,\dots,i_m) = \frac{1}{i_1+i_2+\dots+i_m+1},\quad i_1,\dots,i_m = 0,1,\dots,n-1.
$$
Apparently, a Hilbert tensor is a special Hankel tensor with the generating vector $\big[1,\frac{1}{2},\frac{1}{3},\dots,\frac{1}{mn-m+1}\big]^\top$. Moreover, its associated Hankel matrix is a Hilbert matrix, which is well-known to be positive definite (see, e.g., \cite{Song14}). Thus a Hilbert tensor must be a strong Hankel tensor.
\begin{example}
We take the $4^{\rm th}$-order $5$-dimensional Hilbert tensor, which is generated by $\big[1,\frac{1}{2},\frac{1}{3},\dots,\frac{1}{17}\big]^\top$,  as the second example. Applying Algorithm \ref{alg_vander} and taking $\gamma = \frac{1}{18}$ in the algorithm, we obtain a standard Vandermonde decomposition of
$$
\ten{H} = \sum_{k=1}^9 \alpha_k {\bf v}_k^{\circ 4}, \quad {\bf v}_k = \big[1,\xi_k,\dots,\xi_k^{n-1}\big]^\top,
$$
where $\alpha_k$'s and $\xi_k$'s are displayed in the following table.
\begin{center}
\small
\begin{tabular}{c|rrrrrrrrr}
  \hline
  \hline
  $k$        & 1 & 2 & 3 & 4 & 5 & 6 & 7 & 8 & 9 \\
  \hline
  $\xi_k$ & 0.9841 & 0.9180 & 0.8067 & 0.6621 & 0.5000 & 0.3379 & 0.1933 & 0.0820 & 0.0159 \\
  $\alpha_k$    & 0.0406 & 0.0903 & 0.1303 & 0.1562 & 0.1651 & 0.1562 & 0.1303 & 0.0903 & 0.0406 \\
  \hline
  \hline
\end{tabular}
\end{center}
From the computational result, we can see that a Hilbert tensor is actually a nonnegative strong Hankel tensor with a nonnegative Vandermonde decomposition.
\end{example}

Finally, we shall test some strong Hankel tensors without standard Vandermonde decompositions.
\begin{example}
Randomly generate two real scalars $\xi_1,\xi_2$. Then we construct a $4^{\rm th}$-order $n$-dimensional strong Hankel tensor by
$$
\ten{H} = \begin{bmatrix} 1 \\ 0 \\ \vdots \\ 0 \\ 0 \end{bmatrix}^{\circ 4}
+ \begin{bmatrix} 1 \\ \xi_1 \\ \vdots \\ \xi_1^{n-2} \\ \xi_1^{n-1} \end{bmatrix}^{\circ 4}
+ \begin{bmatrix} 1 \\ \xi_2 \\ \vdots \\ \xi_2^{n-2} \\ \xi_2^{n-1} \end{bmatrix}^{\circ 4}
+ \begin{bmatrix} 0 \\ 0 \\ \vdots \\ 0 \\ 1 \end{bmatrix}^{\circ 4}.
$$
For instance, we set the size $n=10$ and apply Algorithm \ref{alg_vander} to obtain an augmented Vandermonde decomposition of $\ten{H}$. We repeat this experiment for $10000$ times, and the average relative error between the computational solutions $\widehat{\xi}_1,\widehat{\xi}_2$ and the exact solution $\xi_1,\xi_2$ is $4.7895 \times 10^{-12}$. That is, our algorithm recover the augmented Vandermonde decomposition of $\ten{H}$ accurately.
\end{example}

\section{The Third Inheritance Property of Hankel Tensors}

We have proved two inheritance properties of Hankel tensors in this paper.   We now raise the third inheritance property of Hankel tensors as a conjecture.

\bigskip

\noindent {\bf Conjecture}   If a lower-order Hankel tensor has no negative H-eigenvalues, then its associated higher-order Hankel tensor with the same generating vector, where the higher order is a multiple of the lower order, has also no negative H-eigenvalues.

We see that the first inheritance property of Hankel tensors established in Section 2 is only a special case of this inheritance property, i.e., the case that the lower-order Hankel tensor is of even-order.   At this moment, we are not able to prove or to disprove this conjecture if the lower-order Hankel tensor is of odd-order.   However, we may see that if this conjecture is true, then it is of significance.   Assume that this conjecture is true.  If the lower-order Hankel tensor is of odd-order while the higher-order Hankel tensor is of even-order, then we have a new way to identify some positive semi-definite Hankel tensors, and discover a link between odd-order symmetric tensors of no negative H-eigenvalues, and positive semi-definite symmetric tensors.

\begin{acknowledgements}
We would like to thank Prof. Man-Duen Choi and Dr. Ziyan Luo for their helpful discussions. We would also like to thank the editor, Prof. Lars Eld\'{e}n, and the two referees for their detailed and helpful comments.
\end{acknowledgements}



\end{document}